\documentclass[11pt]{amsart}
\usepackage{amsmath,amssymb,color,cite,verbatim}

\numberwithin{equation}{section}

\newtheorem{thm}[equation]{Theorem}
\newtheorem{prop}[equation]{Proposition}
\newtheorem{lemma}[equation]{Lemma}

\theoremstyle{definition}
\newtheorem*{rmk}{Remark}

\newtheorem{defn}[equation]{Definition}

\newcommand{\F}{\mathbb{F}}
\newcommand{\bP}{\mathbb{P}}

\newcommand{\Z}{\mathbb{Z}}

\newcommand{\mybar}[1]{#1\llap{$\overline{\phantom{\rm#1}}$}}

\usepackage[colorlinks,pdftex,pagebackref, bookmarks=false]{hyperref}

\begin{document}

\title{A note on the paper arXiv:2112.14547}

\author{Michael E. Zieve}
\address{
  Department of Mathematics,
  University of Michigan,
  530 Church Street,
  Ann Arbor, MI 48109-1043 USA
}
\email{zieve@umich.edu}
\urladdr{http://www.math.lsa.umich.edu/$\sim$zieve/}

\date{\today}

\begin{abstract}
We give historical remarks related to  arXiv:2112.14547.  In particular, we show that the ``new" permutation polynomials in that paper are actually well known.  In addition we give a simpler derivation of these permutation polynomials than had been given previously, which demonstrates the general method of producing permutation polynomials that was introduced in arXiv:1310.0776.
\end{abstract}

\maketitle


\section{Introduction} 

A polynomial $f(X)\in\F_q[X]$ is called a \emph{permutation polynomial} if the function $c\mapsto f(c)$ permutes $\F_q$.  The recent paper \cite{GWSZL} purports to provide new classes of permutation polynomials.  Here we show that these permutation polynomials are in fact well-known.\footnote{Shortly after \cite{GWSZL1} (the first version of \cite{GWSZL}) appeared on the arXiv, I emailed the content of this note to the third and fifth authors of that paper (I cannot find email addresses for the other authors), and suggested that \cite{GWSZL1} should be revised in light of the content of this note.  Since I did not receive a reply, and no new version of \cite{GWSZL1} was posted in the following days, I posted the first version of this note to the arXiv on 04 Jan 2022, in order to help members of the permutation polynomials community avoid spending further time and effort on rediscovering known results.  The updated version \cite{GWSZL} of \cite{GWSZL1} was posted to the arXiv on 29 Apr 2022, but it did not incorporate the information from the present note or from my emails, although it did remove the claim from \cite{GWSZL1} that that paper was the first to prove three conjectures which had in fact been resolved long ago.}  We also give a new proof of the main result of \cite{GWSZL}, which is significantly simpler and more direct than all previous proofs, and which demonstrates the general method of producing permutation polynomials that was introduced in \cite{ZR}.

The main result of \cite{GWSZL} is as follows.

\begin{thm}\label{them}
Let $k,\ell,m$ be positive integers with $\ell\ne m$, and write $q:=2^k$, $Q:=2^\ell$, and $R:=2^m$.
Let $u$ be an integer, and let $d_1,d_2,d_3$ be positive integers such that
\begin{align*}
d_1 &\equiv Q-R+u(q+1)\pmod{q^2-1} \\
d_2 &\equiv Q+R+(u-R)(q+1)\pmod{q^2-1} \\
d_3&\equiv -(Q+R)+(u+Q)(q+1)\pmod{q^2-1}.
\end{align*}
If $\gcd(d_1,q^2-1)=1$ then $f(X):=X^{d_1}+X^{d_2}+X^{d_3}$ permutes\/ $\F_{q^2}$.
\end{thm}

\begin{rmk}
The above statement includes all permutation polynomials that can be inferred from any interpretation of \cite[Thm.~3.1]{GWSZL}.  The statement of the latter result does not require the $d_i$ to be positive, and does not say that its expressions for the $d_i$'s should be interpreted as congruences mod ($q^2-1$).  However, we assume that the authors of \cite{GWSZL} intended to state their result as above.  The positivity condition is needed in order to make their result be true (and indeed, negative $d_i$'s would not yield polynomials), and after imposing positivity then the congruences becomes natural, since such congruences do not affect whether $f(X)$ permutes $\F_{q^2}$.
\end{rmk}

We will use the following definition.

\begin{defn}
We say that polynomials $f,g\in\F_q[X]$ are \emph{multiplicatively equivalent} if $f(X)\equiv \alpha g(\beta X^n)\pmod{X^q-X}$ for some $\alpha,\beta\in\F_q^*$ and some positive integer $n$ such that $\gcd(n,q-1)$.
\end{defn}

The following properties of multiplicative equivalence are immediate:

\begin{enumerate}
\item Multiplicative equivalence is an equivalence relation on $\F_q[X]$.
\item If $f,g\in\F_q[X]$ are multiplicatively equivalent then $f(X)$ permutes $\F_q$ if and only if $g(X)$ permutes $\F_q$.
\item If $f(X)$ and $g(X)$ are multiplicatively equivalent and $\deg(f)<q$ then $f(X)$ has at most as many terms as does $g(X)$.
\end{enumerate}

In light of the above properties, multiplicative equivalence is a natural equivalence relation to use when deciding whether one permutation polynomial with few terms is essentially the same as another.

\begin{rmk}
What we call multiplicative equivalence has been called ``quasi-multiplicative equivalence" in previous papers.  The term ``multiplicative equivalence" has been defined previously to mean two different things, neither of which is equivalent to the above definition.  However, we suggest that the above definition should be used in the future, for the betterment of the subject -- for instance, the use of the previous definitions has led authors to spend time and effort producing permutation polynomials that could have been obtained immediately from previously known permutation polynomials by composing with $\beta X$.
\end{rmk}

We will show that the permutation polynomials in Theorem~\ref{them}
have appeared in the following previous results (listed according to the order in which the relevant papers were posted in the public domain):

\begin{enumerate}
\item Some instances of the permutation polynomials in Theorem~\ref{them} are special cases of the much more general classes of permutation polynomials in \cite[Thm.~1.1 and 1.2]{ZR}; however, the paper \cite{ZR} does not draw attention to the relevant special cases.
\item If $\ell$ is odd and $m$ is even then every permutation polynomial in Theorem~\ref{them} appears in one of \cite[Cor.~3.8, 3.9, 3.12, and 3.13]{WYDM}.  If $\ell$ is even and $m$ is odd then every permutation polynomial in Theorem~\ref{them} is multiplicatively equivalent to a permutation polynomial in one of \cite[Cor.~3.8, 3.9, 3.12, and 3.13]{WYDM}.
\item Every permutation polynomial in Theorem~\ref{them} is multiplicatively equivalent to a permutation polynomial in \cite[Thm.~1]{LHnewtri}.
\item Every permutation polynomial in Theorem~\ref{them} is multiplicatively equivalent to a permutation polynomial in \cite[Thm.~4.2]{BQ}.
\item Every permutation polynomial in Theorem~\ref{them} is multiplicatively equivalent to a permutation polynomial in \cite[Thm.~3.2]{ZKP}.
\item If $k$ is odd and $\gcd(2k,\ell-m)=1$ then every permutation polynomial in Theorem~\ref{them} is multiplicatively equivalent to a permutation polynomial in \cite[Thm.~1]{LXZ}.
\item Every permutation polynomial in Theorem~\ref{them} is multiplicatively equivalent to a permutation polynomial in \cite[Thm.~1.1]{ZLKPT}.
\end{enumerate}

In addition, \cite{GWSZL1} (the first version of \cite{GWSZL}) purports to be the first paper to resolve three conjectures from the literature.  However, in fact those conjectures were proved previously.  Specifically, the two conjectures from \cite{TZHL} were first resolved as parts (a) and (b) of \cite[Cor.~1.4]{ZR}, and Conjecture~2 of \cite{GS} was first resolved in \cite[Thm.~4.2]{WYDM}.

In the next section we give a very short and simple proof of Theorem~\ref{them}, based on the arguments in \cite{ZR}, which pinpoints the key reason why $f(X)$ permutes $\F_{q^2}$.  This proof avoids the non-conceptual computations occurring in all previous proofs of Theorem~\ref{them}. It turns out that the same approach can be used to deduce all the results mentioned above, in addition to dozens of other results from the literature and arbitrarily many as-yet unpublished results.  We encourage readers interested in permutation polynomials to look into \cite{ZR}, so that they can focus their attention and time on producing results which do not follow immediately from the arguments in that paper.
We conclude this note in Section~3 by explaining how Theorem~\ref{them} connects with previous results.


\section{Proof of Theorem~\ref{them}}

In this section we give a new proof of Theorem~\ref{them}.  We write $\mu_{q+1}$ for the set of $(q+1)$-th roots of unity in $\F_{q^2}$, and if $A(X)\in\F_{q^2}[X]$ then we write $A^{(q)}(X)$ for the polynomial obtained from $A(X)$ by raising all coefficients to the $q$-th power.
We first restate the condition for certain polynomials to permute $\F_{q^2}$ in terms of whether an associated polynomial permutes $\mu_{q+1}$, via the following special case of an easy and  much-used lemma from \cite{Zlem}.

\begin{lemma}\label{old}
Write $f(X):=X^r A(X^{q-1})$ where $r$ is a positive integer, $q$ is a prime power, and $A(X)\in\F_{q^2}[X]$.  Then $f(X)$ permutes\/ $\F_{q^2}$ if and only if $\gcd(r,q-1)=1$ and $g_0(X):=X^r A(X)^{q-1}$ permutes $\mu_{q+1}$.
\end{lemma}

We next translate the condition that $g_0(X)$ permutes $\mu_{q+1}$ into the condition that an associated rational function $g(X)$ permutes $\mu_{q+1}$, where typically $g(X)$ has much lower degree than does $g_0(X)$.  We do this in the following trivial lemma, which encodes a procedure introduced in \cite{ZR}.

\begin{lemma}\label{rewrite}
Write $g_0(X):=X^r A(X)^{q-1}$ where $r$ is an integer, $q$ is a prime power, and $A(X)\in\F_{q^2}[X]$.  Then $g_0(X)$ maps $\mu_{q+1}$ into $\mu_{q+1}\cup\{0\}$, and if $A(X)$ has no roots in $\mu_{q+1}$ then $g_0(X)$ induces the same function on $\mu_{q+1}$ as does $g(X):=X^s A^{(q)}(1/X)/A(X)$ for any integer $s$ with $s\equiv r\pmod{q+1}$.  In particular, $g_0(X)$ permutes $\mu_{q+1}$ if and only if $g(X)$ permutes $\mu_{q+1}$ and $A(X)$ has no roots in $\mu_{q+1}$.
\end{lemma}

A key ingredient in \cite{ZR} is degree-one rational functions which map $\mu_{q+1}$ to either $\mu_{q+1}$ or $\bP^1(\F_q):=\F_q\cup\{\infty\}$.  In this note we use
\[
\rho(X):=\frac{X+\omega}{\omega X+1}
\]
where $\omega$ is a prescribed order-$3$ element of $\F_{q^2}^*$, with $q$ being a power of $2$. 
The following result is a special case of \cite[Lemmas~2.1 and 3.1]{ZR}, and also is easy to verify directly.

\begin{lemma}\label{rho}
Let $q:=2^k$ where $k>0$.  If $k$ is even then $\rho(X)$ permutes $\mu_{q+1}$, and if $k$ is odd then $\rho(X)$ interchanges $\mu_{q+1}$ and\/ $\bP^1(\F_q)$.
\end{lemma}

%
%

Pick any nonconstant $h(X)\in\mybar\F_q(X)$.  In light of Lemma~\ref{rho}, if $k$ is even then $\rho\circ h\circ\rho$ permutes $\mu_{q+1}$ if and only if $h(X)$ permutes $\mu_{q+1}$, and if $k$ is odd then $\rho\circ h\circ\rho$ permutes $\mu_{q+1}$ if and only if $h(X)$ permutes $\bP^1(\F_q)$.  We will show that the permutation polynomials $f(X)$ in Theorem~\ref{them} correspond to rational functions $g(X)$ permuting $\mu_{q+1}$ (via Lemmas~\ref{old} and \ref{rewrite}) where $g(X)=X^i\circ\rho\circ X^{R+jQ}\circ\rho$ with $i,j\in\{1,-1\}$.  The following result presents these compositions in the cases we need; it can be verified by a routine computation.

\begin{lemma}\label{Q1} 
Let $k,\ell,m$ be positive integers with $\ell\ne m$, and write $q:=2^k$, $Q:=2^\ell$, and $R:=2^m$. Then
\[
X^{(-1)^m}\circ\frac{X^{Q+R}+X^Q+1}{X^{Q+R}+X^R+1}
=\begin{cases}
\rho\circ X^{R-Q}\circ\rho &\text{ if $\ell\equiv m\pmod{2}$} \\
\rho\circ X^{R+Q}\circ\rho &\text{ if $\ell\not\equiv m\pmod{2}$.}
\end{cases}
\]
\end{lemma}

Now we prove Theorem~\ref{them}.

\begin{proof}[Proof of Theorem~\ref{them}]
Note that $d_1\equiv d_2+R(q-1)\pmod{q^2-1}$ and $d_3\equiv d_2+(Q+R)(q-1)\pmod{q^2-1}$.
Thus $d_2\equiv d_1\pmod{q-1}$, so the hypothesis $\gcd(d_1,q-1)=1$ implies that $\gcd(d_2,q-1)=1$.
By Lemma~\ref{old} and Lemma~\ref{rewrite}, it suffices to show that $A(X):=X^R+1+X^{Q+R}$ has no roots in $\mu_{q+1}$ and $g(X)$ permutes $\mu_{q+1}$, where
\[
g(X):=X^{Q+R} \frac{A^{(q)}(1/X)}{A(X)}=\frac{X^{Q+R}+X^Q+1}{X^{Q+R}+X^R+1}.
\]

We first show that $A(X)$ has no roots in $\mu_{q+1}$.  Suppose to the contrary that $\alpha\in\mu_{q+1}$ satisfies $A(\alpha)=0$.  Then also
\[
0=\alpha^{Q+R} A(\alpha)^q = \alpha^{Q+R}A(\alpha^q) =
\alpha^{Q+R}A\bigl(\frac1{\alpha}\bigr)=\alpha^Q+\alpha^{Q+R}+1.
\]
Thus $0=A(\alpha)+\alpha^{Q+R}A(\alpha)^q=\alpha^R+\alpha^Q$, so $\alpha^{Q-R}=1$.  Since $\gcd(Q-R,q+1)=1$, it follows that $\alpha=1$; but plainly $A(1)=1\ne 0$, contradiction.

It remains to show that $g(X)$ permutes $\mu_{q+1}$. 
First suppose $\ell\equiv m\pmod{2}$.  Then the hypothesis $\gcd(Q-R,q+1)=1$ implies that $k$ is even, so that $X^{R-Q}$ permutes $\mu_{q+1}$ and also $\rho(X)$ permutes $\mu_{q+1}$ by Lemma~\ref{rho}.  Thus Lemma~\ref{Q1} implies that $g(X)$ permutes $\mu_{q+1}$

Now suppose $\ell\not\equiv m\pmod{2}$.  If $k$ is odd then $\rho(X)$ interchanges $\mu_{q+1}$ and $\bP^1(\F_q)$, and we have $\gcd(Q+R,q-1)=1$ so that $X^{Q+R}$ permutes $\bP^1(\F_q)$, whence $g(X)$ permutes $\mu_{q+1}$ by Lemma~\ref{Q1}.  Finally, if $k$ is even then $\rho(X)$ permutes $\mu_{q+1}$, and since $k\not \equiv \ell-m\pmod{2}$ we have $\gcd(Q+R,q+1)=1$, so that $X^{Q+R}$ permutes $\mu_{q+1}$, whence again $g(X)$ permutes $\mu_{q+1}$.
\end{proof}

\begin{rmk}
The method used in the above proof can be used to produce enormous collections of permutation polynomials over $\F_{q^2}$, for any prime power $q$.  One can start with any rational function $h(X)\in\F_q(X)$ which permutes $\bP^1(\F_q)$, and any degree-one $\rho,\eta\in\F_{q^2}(X)$ such that $\rho(\mu_{q+1})=\bP^1(\F_q)$ and $\eta(\bP^1(\F_q))=\mu_{q+1}$, in order to obtain a rational function $g(X):=\eta\circ h\circ\rho$ which permutes $\mu_{q+1}$.  It turns out that $g(X)$ can always be written in infinitely many ways as $X^s A^{(q)}(1/X)/A(X)$ where $s\in\Z$ and $A(X)\in\F_{q^2}[X]$ has no roots in $\mu_{q+1}$.  If either $q$ is even or $s$ is odd then there exist positive integers $r$ such that $r\equiv s\pmod{q+1}$ and $\gcd(r,q-1)=1$, so that $X^r A(X^{q-1})$ permutes $\F_{q^2}$.  By applying this procedure to the most well-known permutation rational functions over $\F_q$, and using certain choices of $\rho(X)$, $\eta(X)$, and $A(X)$, one obtains huge classes of permutation polynomials over $\F_{q^2}$ which include as very special cases essentially all known permutation polynomials of the form $X^r B(X^{q-1})$.  We will elaborate on this remark in forthcoming joint papers with Zhiguo Ding.
\end{rmk}


\section{Connection with previous results}

In this section we explain the connection between Theorem~\ref{them} and previous results.  The combination of \cite[Cor.~3.8, 3.9, 3.12, and 3.13]{WYDM} is as follows.

\begin{prop}\label{WYDM}
Let $k,s,t,r$ be positive integers with $s$ odd and $t$ even, and write $q:=2^k$, $S:=2^s$, and $T:=2^t$.  If $r\equiv S+T\pmod{q+1}$ then
\[
g(X):=X^r\Bigl(X^{(S+T)(q-1)}+X^{T(q-1)}+1\Bigr)
\]
permutes\/ $\F_{q^2}$ if and only if $\gcd(r,q-1)=1$.
\end{prop}

This implies Theorem~\ref{them} in case $\ell$ is odd and $m$ is even, since if we put $s:=\ell$, $t:=m$, and $r:=d_2$ then the polynomial $g(X)$ in Proposition~\ref{WYDM} is congruent mod $X^{q^2}-X$ to the polynomial $f(X)$ in Theorem~\ref{them}.
Next suppose that $\ell$ is even and $m$ is odd, and put $s:=m$, $t:=\ell$, and $r\equiv -d_3-(Q+R)(q-1)\pmod{q^2-1}$.  Then one can check that the polynomials $g(X)$ from Proposition~\ref{WYDM} and $f(X)$ from Theorem~\ref{them} satisfy 
\[
f(X)\equiv g(X^{q^2-2})\pmod{X^{q^2}-X},
\]
so that $f(X)$ and $g(X)$ are multiplicatively equivalent.
%
%

Theorem~1 of \cite{LHnewtri} is as follows:

\begin{prop}\label{LH}
Let $k$ and $n$ be positive integers, and write $q:=2^k$ and $T:=2^n$.  Suppose that $\gcd(T-1,q+1)=1$, and let $r$ and $s$ be positive integers such that $r(T-1)\equiv T\pmod{q+1}$ and $s(T-1)\equiv -1\pmod{q+1}$.  Then $g(X):=X+X^{1+r(q-1)}+X^{1+s(q-1)}$ permutes\/ $\F_{q^2}$.
\end{prop}

\begin{rmk}
The statement of \cite[Thm.~1]{LHnewtri} has the additional hypothesis $n<k$, but that hypothesis is not used in the proof of that result.
\end{rmk}

We now show that all the permutation polynomials in Theorem~\ref{them} are multiplicatively equivalent to permutation polynomials in Proposition~\ref{LH}.  Assume the hypotheses of Theorem~\ref{them}.  Replace $\ell$ by $\ell+2ki$ for a positive integer $i$ which is large enough so that $Q>R$; note that this replacement does not change the congruence class of $Q$ mod $(q^2-1)$, and hence does not affect the truth of the hypotheses or conclusion of Theorem~\ref{them}, while also not affecting the multiplicative equivalence class of the permutation polynomial $f(X)$ in Theorem~\ref{them}.
Write $n:=\ell-m$ and $T:=2^n$, so that the hypotheses of Theorem~\ref{them} imply that $\gcd(T-1,q+1)=1$.  Writing $v:=d_1$, we see that the polynomial $g(X)$ in Proposition~\ref{LH} satisfies \[
g(X^v)=X^v+X^{v+rv(q-1)}+X^{v+sv(q-1)}.
\]
Since $v\equiv Q-R\pmod{q+1}$, we have
\begin{align*}
rv(q-1)&\equiv r(Q-R)(q-1)\pmod{q^2-1}\\
&=r(T-1)R(q-1)\\
&\equiv TR(q-1)\pmod{q^2-1}\\
&=Q(q-1),
\end{align*}
and likewise
\[
sv(q-1)\equiv -R(q-1)\pmod{q^2-1}.
\]
It follows that the polynomial $f(X)$ in Theorem~\ref{them} satisfies 
\[
g(X^v)\equiv f(X)\pmod{X^{q^2}-X},
\]
so that $f(X)$ and $g(X)$ are multiplicatively equivalent.

Each of the results \cite[Thm.~4.2]{BQ}, \cite[Thm.~3.2]{ZKP}, and \cite[Thm.~1.1]{ZLKPT} generalizes \cite[Thm.~1]{LHnewtri}, and hence includes special cases that are multiplicatively equivalent to each of the permutation polynomials in Theorem~\ref{them}.  If $k$ is odd and $\gcd(2k,\ell-m)=1$ then the same is true of \cite[Thm.~1]{LXZ}.



\end{document}